\documentclass[preprint,a4paper]{elsarticle}
\usepackage{fullpage}
\usepackage{elf}

\journal{Computers and Mathematics with Applications}

\begin{document}
  \begin{frontmatter}
    \author[1,2]{Erich L Foster\corref{cor1}}
    \ead{efoster@bcamath.org}
    \ead[url]{http://www.math.vt.edu/people/erichlf}
    \cortext[cor1]{corresponding author}

    \author[2]{Traian Iliescu}
    \ead{iliescu@vt.edu}
    \ead[url]{http://www.math.vt.edu/people/iliescu}

    \author[2]{David R. Wells}
    \ead{drwells@vt.edu}
    \ead[url]{http://www.math.vt.edu/people/drwells}

    \address[1]{Basque Center for Applied Mathematics, Alameda Mazarredo, 14,
      48009 Bilbao, Basque Country - Spain}
    \address[2]{Department of Mathematics, Virginia Tech, Blacksburg, VA
      24061-0123, U.S.A.}

    \title{
      A Two-Level Finite Element Discretization of the Streamfunction
      Formulation of the Stationary Quasi-Geostrophic Equations of the Ocean
    }


    \begin{abstract}
      In this paper we proposed a two-level finite element discretization of the
      nonlinear stationary quasi-geostrophic equations, which model the wind
      driven large scale ocean circulation. Optimal error estimates for the
      two-level finite element discretization were derived. Numerical experiments
      for the two-level algorithm with the Argyris finite element were also
      carried out. The numerical results verified the theoretical error estimates
      and showed that, for the appropriate scaling between the coarse and fine
      mesh sizes, the two-level algorithm significantly decreases the
      computational time of the standard one-level algorithm.
    \end{abstract}

  \end{frontmatter}

  \section{Introduction} \label{sec:Intro}
  Two-level algorithms are computationally efficient approaches for \emph{finite
element} (FE) discretizations of nonlinear partial differential equations
\cite{Borggaard08,Borggaard12,Fairag98,Layton93,Xu94}. A two-level FE discretization
aims to solve a particular nonlinear elliptic equation by first solving the
nonlinear system on a coarse mesh and then using the coarse mesh solution to
solve the linearized system on a fine mesh.  The appeal of such a method is
clear; one need only solve the nonlinear equations on a coarse mesh and then use
this solution to solve on a fine mesh, thereby reducing computational time
without sacrificing solution accuracy. The development of the two-level FE
discretization was originally performed by Xu in \cite{Xu94}. Later algorithms
were developed for the \emph{Navier-Stokes equations} (NSE) by Layton
\cite{Layton93} (see also \cite{Fairag98, Fairag03, Shao11, Ye98, Ye99AMC,
Ye99NFAO}) and for the Boussinesq equations by Lenferink \cite{Lenferink94}.

As computational power increases, complex models are becoming more and more
popular for the numerical simulation of oceanic and atmospheric flows.
Computational efficiency, however, remains an important consideration for
geophysical flows in which long time integration is needed. Thus, simplified
mathematical models are central to the numerical simulation of such flows.  For
example, the \emph{quasi-geostrophic equations} (QGE), a standard mathematical
model for wind driven large scale oceanic and atmospheric flows
\cite{Majda,Vallis06}, are often used in climate modeling \cite{Dijkstra05}.

Most FE discretizations of the QGE are for the streamfunction-vorticity
formulation.  The reason is that the streamfunction-vorticity formulation allows
the use of low order ($C^0$) FEs, although one needs to discretize two flow
variables, the potential vorticity, $q$, and the streamfunction, $\psi$. We note
that the streamfunction-vorticity formulation is often used in the numerical
discretization of the 2D NSE, to which the QGE are similar in form.  Alternatively, one
can, instead, use the pure streamfunction formulation of the QGE. The advantage
lies in an equation that contains only one flow variable, the streamfunction,
$\psi$, at the price of having to deal with a fourth-order partial differential
equation. Thus, the numerical discretization of the pure streamfunction formulation of the QGE with conforming FEs requires the
use of high-order ($C^1$) FEs, e.g., the Argyris element~\cite{Argyris,Brenner}.

The streamfunction formulation of the QGE still suffers from having to solve a
large nonlinear system of equations. This is usually done by using a nonlinear
solver, such as Newton's method.  These nonlinear solvers typically require
solving large linear systems multiple times to obtain the solution to the
nonlinear system. Solving these large linear systems multiple times can be time
consuming. Thus, a two-level algorithm can significantly reduce computational time over
the standard nonlinear solver, since we need only solve the nonlinear system on a
coarse mesh and then use that solution to solve a linear system on a fine mesh.

In this paper, we propose a two-level algorithm for the FE discretization of the
\emph{streamfunction formulation of the stationary QGE} (SQGE). 
Just as in the NSE case \cite{Girault86,Temam,layton2008introduction}, we regard the stationary QGE as a stepping-stone to the time-dependent QGE, which are the ultimate goal of the two-level algorithm put forth in this report. 
The conforming FE
discretization is based on the Argyris element. Additionally, we present a
rigorous error analysis for the two-level FE discretization. The theoretical
error bounds as well as the increased computational efficiency are illustrated
numerically for two test problems.

The rest of the paper is organized as follows: In \autoref{sec:Stream} we
present the SQGE. In \autoref{sec:WeakForm} we present the weak formulation
of the SQGE, including notation and functional spaces. \autoref{sec:FEForm} contains the presentation of the one-level FE
discretization of the SQGE. In \autoref{sec:TwoLevel}, we discuss both the
two-level algorithm and its application to the SQGE. Next, in
\autoref{sec:Errors} we provide rigorous error bounds for the two-level FE
discretization of the SQGE and we discuss the scaling between the fine mesh size,
$h$, and the coarse mesh size, $H$. \autoref{sec:Numerical} includes numerical
results which both verify the theoretical error bounds presented in
\autoref{sec:Errors} and illustrate the computational efficiency of the two-level algorithm over the standard one-level method.
Finally, in \autoref{sec:Conclusions} we present our
conclusions.

  \section{Streamfunction Formulation} \label{sec:Stream}
  The SQGE in a simply connected domain $\Omega$ is
\begin{equation}
  Re^{-1} \Delta^2 \psi + J(\psi,\Delta\psi) - Ro^{-1}\frac{\partial
    \psi}{\partial x} = Ro^{-1} F \quad \text{ in } \Omega,
  \label{eqn:Streamfunction}
\end{equation}
where
\begin{equation}
  J(u,v) = \frac{\partial u}{\partial x}\frac{\partial v}{\partial y} -
    \frac{\partial u}{\partial y} \frac{\partial v}{\partial x}, \\
  \label{eqn:Jacobian}
\end{equation}
\begin{equation}
    Ro = \frac{U}{\beta L^2},
    \label{eqn:Rossby}
\end{equation}
\begin{equation}
  Re^{-1} = \frac{U L}{A}
  \label{eqn:Reynolds}
\end{equation}
are the Jacobian, Rossby number, Reynolds number, respectively, and $\beta$,
$A$, $U$, and $L$ are the coefficient in the beta plane approximation, the
eddy viscosity, the characteristic velocity scale, and the characteristic
length scale, respectively (see \cite{Foster, San11}).

To completely specify \eqref{eqn:Streamfunction}, we need to impose boundary
conditions
(see \cite{Cummins, Vallis06, San12} for a careful discussion of this issue). In this
report, we consider
\begin{equation}
  \psi = \frac{\partial \psi}{\partial \vec{n}} = 0 \quad \text{on } \Omega,
  \label{eqn:BCs}
\end{equation}
where $\vec{n}$ represents the outward unit normal to $\Omega$. These are also
the boundary conditions used in \cite{Gunzburger89, Fairag98, Fairag} for
the 2D NSE.

  \section{Weak Formulation} \label{sec:WeakForm}
  Now we can derive the weak formulation of the SQGE
\eqref{eqn:Streamfunction}. To this end, we first introduce the appropriate
functional setting. Let
\begin{equation*}
  X:=H^2_0(\Omega)= \left\{ \psi\in H^2(\Omega):
    \psi=\frac{\partial\psi}{\partial \vec{n}}=0 \text{ on } \partial\Omega \right\}.
\end{equation*}
Multiplying \eqref{eqn:Streamfunction} by a test function $\chi \in X$ and
using the divergence theorem, we get, in the standard way (see
\cite{Gunzburger89}), the \emph{weak formulation} of the SQGE:
\begin{equation}
  \begin{split}
    \text{\emph{Find }} \psi \in X \text{\emph{ such that }}
    \qquad \qquad \qquad \\
    a(\psi,\chi) + b(\psi;\psi,\chi) + c(\psi,\chi) = \ell(\chi),\quad \forall
      \chi \in X,
  \end{split}
  \label{eqn:WeakForm}
\end{equation}
where
\begin{equation}
  \begin{split}
    a(\psi,\chi) &= Re^{-1} \int_{\Omega}\! \Delta \psi \Delta \chi \,d\vec{x}, \\
    b(\zeta;\psi,\chi) &= \int_{\Omega}\! \Delta \zeta\, \left( \psi_y\chi_x
      - \psi_x\chi_y \right) \,d\vec{x}, \\
    c(\psi,\chi) &= -Ro^{-1} \int_{\Omega}\! \psi_x \chi \,d\vec{x}, \\
    \ell(\chi) &= Ro^{-1} \int_{\Omega}\! F \, \chi \,d\vec{x}.
  \end{split}
  \label{eqn:Forms}
\end{equation}
We note that in the space $H^2_0$ the semi-norm $|\cdot|_2$ and the norm
$\|\cdot\|_2$ are equivalent (see (1.2.8) in \cite{Ciarlet}).

\begin{lemma} \label{lma:bounds}
  Given $\psi,\, \xi,\, \varphi \in H^2_0(\Omega)$ and $F\in H^{-2}(\Omega)$,
  the linear form $\ell$, the bilinear forms $a$ and $c$, and the trilinear form
  $b$ are continuous: there exist $\Gamma_1 > 0$ and $\Gamma_2 > 0$ such that
  \begin{align}
    a(\psi,\chi) &\le Re^{-1} |\psi|_2 |\chi|_2 \label{eqn:aCont} \\
    b(\zeta;\psi,\chi) &\le \Gamma_1 |\zeta|_2 |\psi|_2 |\chi|_2 \label{eqn:bCont} \\
    c(\psi,\chi) &\le Ro^{-1}\, \Gamma_2 |\psi|_2 |\chi|_2 \label{eqn:cCont} \\
    \ell(\chi) &\le Ro^{-1} \|F\|_{-2} |\chi|_2. \label{eqn:lCont}
  \end{align}
\end{lemma}
For a proof, see \cite{Cayco86}.

%
For small enough data, one can use the same type of arguments as those used in Chapter 6 in~\cite{layton2008introduction} (see also \cite{Girault79,Girault86}) to prove that the SQGE \eqref{eqn:Streamfunction} are well-posed \cite{barcilon1988existence,wolansky1988existence}.
In what follows, we will always assume that the small data condition involving $Re$, $Ro$ and $F$, is satisfied and, thus, that there exists a unique solution $\psi$ to \eqref{eqn:Streamfunction}.

The following stability estimate was proven in Theorem 1 in~\cite{Foster}:
\begin{lemma} \label{lma:Stability}
  The solution $\psi$ of \eqref{eqn:Streamfunction} satisifies the following
  stability estimate:
  \begin{equation}
    |\psi|_2 \le Re\, Ro^{-1} \|F\|_{-2}.
    \label{eqn:psi}
  \end{equation}
\end{lemma}

  \section{Finite Element Formulation} \label{sec:FEForm}
  Let $\mathcal{T}^H$ denote a FE triangulation of $\Omega$ with
meshsize (maximum triangle diameter) $H$.  We consider a {\it conforming} FE
discretization of \eqref{eqn:WeakForm}, i.e., $X^H \subset X = H_0^2(\Omega)$.

The FE discretization of the SQGE \eqref{eqn:WeakForm} reads: Find $\psi^H \in
X^H$ such that
\begin{eqnarray}
    a(\psi^H,\chi^H) + b(\psi^H;\psi^H,\chi^H) + c(\psi^H,\chi^H) = \ell(\chi^H),\quad \forall \,
      \chi^H \in X^H.
    \label{eqn:FEForm}
\end{eqnarray}
Using standard arguments \cite{Girault79,Girault86}, one can prove that, if the small data condition used in proving the well-posedness result for the continuous case holds, then \eqref{eqn:FEForm} has a unique solution $\psi^H$ (see Theorem 2.1 and subsequent discussion in \cite{Cayco86}).
Furthermore, one can prove the following
stability result for $\psi^H$ using the same arguments as those used in the
proof of Lemma \ref{lma:Stability} for the continuous setting (see Theorem 2 in~\cite{Foster}).
\begin{lemma} \label{lma:FEStability}
  The solution $\psi^H$ of \eqref{eqn:FEForm} satisfies the following stability
  estimate:
  \begin{equation}
    |\psi^H|_2
    \le Re \, Ro^{-1} \, \| F \|_{-2} .
    \label{eqn:FEStability}
  \end{equation}
\end{lemma}

As noted in Section 6.1 in \cite{Ciarlet} (see also Section 13.2 in
\cite{Gunzburger89}, Section 3.1 in \cite{Johnson}, and Theorem 5.2 in
\cite{Braess}), in order to develop a conforming FE for the SQGE
\eqref{eqn:WeakForm}, we are faced with the problem of constructing subspaces of
$H^2_0(\Omega)$. Since the standard, piecewise polynomial FE spaces are locally
regular, this construction amounts in practice to finding FE spaces $X^H$ that
satisfy the inclusion $X^H \subset C^1({\bar \Omega})$, i.e., finding $C^1$ FEs.
Several FEs meet this requirement (see, e.g., Section 6.1 in \cite{Ciarlet},
Section 13.2 in \cite{Gunzburger89}, and Section 5 in \cite{Braess}): the
Argyris triangular element, the Bell triangular element, the Hsieh-Clough-Tocher
triangular element (a macroelement), and the Bogner-Fox-Schmit rectangular
element. We note that any $C^1$ FE can be used for this study. We chose  
the Argyris element, however, because it is a triangle, which allows for easy treatment
of complex boundaries, and because of the recent development of a transformation
that allows for calculations to be done on the reference element
\cite{Dominguez08}. We also note that, in addition to the use of $C^1$ conforming FEs one might use
the newly developed method of isogeometric analysis as in \cite{Auricchio2007}.

  \section{Two-Level Algorithm} \label{sec:TwoLevel}
  In this section we propose a two-level FE discretization of the SQGE
\eqref{eqn:WeakForm}. We let $X^h,\, X^H \subset X=H^2_0(\Omega)$ denote two
conforming FE meshes with $H > h$. The two-level algorithm consists of two
steps. In the first step, the nonlinear system is solved on a coarse mesh, with
mesh size $H$.  In the second step, the nonlinear system is linearized around
the approximation found in the first step, and the resulting linear system is
solved on the fine mesh, with mesh size $h$. This procedure is as follows:
\begin{algorithm}[H]
  \caption{Two-Level algorithm}
  \label{alg:TwoLevel}
  \begin{enumerate}[Step 1:]
    \item Solve the following nonlinear system on a coarse mesh for $\psi^H\in X^H$:
    \begin{equation}
      a(\psi^H,\chi^H) + b(\psi^H; \psi^H,\chi^H) + c(\psi^H,\chi^H) = \ell(\chi^H), \quad \text{for all }
        \chi^H \in X^H.
      \label{eqn:Coarse}
    \end{equation}
    \item Solve the following linear system on a fine mesh for $\psi^h\in X^h$:
    \begin{equation}
      a(\psi^h,\chi^h) + b(\psi^H; \psi^h,\chi^h) + c(\psi^h,\chi^h) = \ell(\chi^h), \quad \text{for all }
        \chi^h \in X^h.
      \label{eqn:Fine}
    \end{equation}
  \end{enumerate}
\end{algorithm}

The well-posedness of the nonlinear system~\eqref{eqn:Coarse} was proven in \cite{Cayco86} (see
also \cite{Foster}).  
The following error estimate for the approximation in Step 1
of the two-level algorithm (\autoref{alg:TwoLevel}) was proven in Theorem 2 in~\cite{Foster}:
\begin{thm} \label{thm:EnergyNorm}
  Assume that the following small data condition is satisfied:
  \begin{equation*}
    Re^{-2}\,Ro \ge \Gamma_1 \|F\|_{-2}.
  \end{equation*}
  Let $\psi$ be the solution of \eqref{eqn:WeakForm} and $\psi^H$ be the
  solution of \eqref{eqn:Coarse}. Then the following error estimate holds:
  \begin{equation}
    |\psi - \psi^H|_2 \le C(Re,Ro,\Gamma_1,\Gamma_2,F) \inf_{\chi^H \in X^H}
      |\psi - \chi^H|_2,
    \label{eqn:EnergyNorm}
  \end{equation}
  where
  \begin{equation*}
    C(Re,Ro,\Gamma_1,\Gamma_2,F):=\frac{\Gamma_2\,
      Ro^{-1}+2Re^{-1}+\Gamma_1\,Re\,Ro^{-1} \|F\|_{-2}}{Re^{-1}-\Gamma_1\, Re\,
      Ro^{-1} \|F\|_{-2}}.
  \end{equation*}
\end{thm}

The following lemma proves the well-posedness of the linear system \eqref{eqn:Fine}:
\begin{lemma}\label{lma:Fine}
  Given a solution $\psi^H$ of \eqref{eqn:Coarse}, the solution $\psi^h$ of
  \eqref{eqn:Fine} exists uniquely.
\end{lemma}
\begin{proof}
  First, we introduce the bilinear form $B:\, X^h \times X^h \to \R$ given by
  \begin{equation}
    B(\psi^h,\chi^h) = a(\psi^h,\chi^h) + b(\psi^H;\psi^h,\chi^h) +
      c(\psi^h,\chi^h).
      \label{eqn:fineBilinear}
  \end{equation}
  \autoref{lma:bounds} yields the following inequality:
  \begin{equation}
    B(\psi^h,\chi^h) \le \left(Re^{-1} + \Gamma_1 |\psi^H|_2 + Ro^{-1}\,
      \Gamma_2\right)|\psi^h|_2\, |\chi^h|_2, \quad \forall \psi^h, \chi^h \in
      X^h.
      \label{eqn:Bineq}
  \end{equation}
  The stability estimate for $\psi^H$ in \autoref{lma:FEStability} and
  inequality \eqref{eqn:Bineq} imply that $B$ is continuous. Additionally, the
  fact that $b(\psi^H; \psi^h, \psi^h) = 0$ and $c(\psi^h,\psi^h) = 0$ for all
  $\psi^h \in X^h$ combined with the Poincar\'e-Friedrichs inequality gives
  \begin{equation*}
    B(\psi^h,\psi^h) \ge C\,\|\psi^h\|_2, \quad \forall \psi^h \in X^h.
  \end{equation*}
  Thus, $B$ is coercive. Therefore, by the Lax-Milgram lemma, $\psi^h$ exists
  and is unique.
\end{proof}

In addition to the existence, uniqueness, and stability of the solution of the
continuous linear system (\autoref{lma:Fine}) we also have a stability bound for
the solution on the discrete fine mesh, $h$.
\begin{lemma} \label{lma:2LStability}
  The solution $\psi^h$ of \eqref{eqn:Fine} satisfies the following stability
  bound:
  \begin{equation}
    \|\psi^h\|_2 \le Re\, Ro^{-1} \|F\|_{-2}.
    \label{eqn:2LStab}
  \end{equation}
\end{lemma}
\begin{proof}
  Setting $\chi^h=\psi^h$ in \eqref{eqn:Fine}, and noting that $c(\psi^h,\chi^h)
  = -c(\chi^h,\psi^h)$, which implies that $c(\psi^h,\psi^h) = 0$, and
  $b(\psi^H;\psi^h,\psi^h) = 0$ gives
  \begin{align*}
    Re^{-1} \|\psi^h\|_2^2 &= \ell(\psi^h) \Rightarrow \\
    \|\psi^h\|_2 &= Re \frac{\ell(\psi^h)}{\|\psi^h\|_2} \le Re Ro^{-1}
      \|F\|_{-2},
  \end{align*}
  where in the last inequality we used \eqref{eqn:lCont}. Therefore, it follows
  that
  \begin{equation*}
    \|\psi^h\|_2 \le Re\, Ro^{-1} \|F\|_{-2}. \qedhere
  \end{equation*}
\end{proof}

  \section{Error Bounds} \label{sec:Errors}
  The main goal of this section is to develop a rigorous numerical analysis for
the two-level algorithm (\autoref{alg:TwoLevel}). The proof for the error bounds
follows the pattern used in \cite{Fairag98}.

We first introduce an improved bound on the trilinear form $b(\zeta; \xi,
\chi)$.  
To this end, we use the following inverse inequality, proven in Lemma 6.4 in~\cite{thomee2006galerkin} (see also page 122 in \cite{Fairag98}):
\begin{equation}
  \|\nabla \varphi^h\|_{L^{\infty}} \le c \sqrt{|\ln(h)|}\, |\varphi^h|_2\quad \forall \varphi^h \in X^h.
  \label{eqn:SobolevIneq}
\end{equation}

The following lemmas will be useful in determining the error bounds for Step
2 of the two-level algorithm. The first lemma, which corresponds to
Lemma 5.1 in \cite{Fairag98}, follows from \eqref{eqn:SobolevIneq} and \eqref{eqn:bCont}
and places error bounds on the trilinear form $b(\psi; \chi^h, \xi)$:
\begin{lemma} \label{lma:bImproved}
  For any $\chi^h\in X^h$, the following inequality holds:
  \begin{align*}
    |b_0(\psi;\chi^h,\xi)| \le C\sqrt{|\ln(h)|} \, |\psi|_2 |\xi|_1 |\chi^h|_2,
  \end{align*}
  where
  \begin{equation}
    b_0(\xi; \chi, \psi) = \int_{\Omega} (\xi_y \chi_{xy} - \xi_x \chi_{yy})
    \psi_y - (\xi_x \chi_{xy} - \xi_y \chi_{xx}) \psi_x d\vec{x}
    \label{eqn:trilinear}
  \end{equation}
\end{lemma}
The following lemma, which corresponds to Lemma 5.6 in \cite{Fairag98}, will
be useful for proving the error bounds for \autoref{alg:TwoLevel}, by
allowing one to permute the terms of the trilinear form:
\begin{lemma} \label{lma:trilinear}
  For $\psi,\,\xi,\,\chi\in H^2_0(\Omega)$, we have
  \begin{equation}
    b(\psi; \xi, \chi) = b_0(\xi; \chi, \psi) - b_0(\chi; \xi, \psi).
    \label{eqn:eqn:trilinear}
  \end{equation}
\end{lemma}

The following theorem gives the error bound after Step 2 of the two-level
algorithm (\autoref{alg:TwoLevel}) and is the main result of this paper. The
proof of this theorem is similar to the proof of Theorem 5.2 in \cite{Fairag98}.
\begin{thm} \label{thm:TwoLevel}
  Let $\psi$ be the solution of \eqref{eqn:WeakForm} and $\psi^h$ the solution
  of \eqref{eqn:Fine}. Then $\psi^h$ satisfies
  \begin{equation}
    |\psi-\psi^h|_2 \le C_1 \inf_{\lambda^h\in X^h} |\psi-\lambda^h|_2 + C_2
      \sqrt{|\ln h|}\, |\psi - \psi^H|_1,
    \label{eqn:Error}
  \end{equation}
  where $C_1 = 2 + Re\,Ro^{-1}\, \Gamma_2 + Re^2 Ro^{-1} \Gamma_1 \|F\|_{-2}$
  and $C_2= 2 Re^2 Ro^{-1} C\,\|F\|_{-2}$.
\end{thm}
\begin{proof}
  Subtracting \eqref{eqn:Fine} from \eqref{eqn:WeakForm} and letting
  $\chi=\chi^h \in X^h \subset X$ yields the error equation:
  \begin{equation}
    a(\psi - \psi^h,\chi^h) + b(\psi;\psi,\chi^h) - b(\psi^H;\psi^h,\chi^h)
      + c(\psi-\psi^h,\chi^h) = 0, \quad \forall \chi^h \in X^h.
      \label{eqn:ErrorEqn}
  \end{equation}
  Now, adding the terms
  \begin{equation*}
    -b(\psi;\psi^h,\chi^h) + b(\psi;\psi^h,\chi^h)
  \end{equation*}
  to \eqref{eqn:ErrorEqn} gives
  \begin{equation}
    a(\psi - \psi^h,\chi^h)
      + b(\psi;\psi-\psi^h,\chi^h) + b(\psi-\psi^H;\psi^h,\chi^h)
      + c(\psi-\psi^h,\chi^h) = 0, \quad \forall \chi^h \in X^h.
      \label{eqn:ErrorEqn2}
  \end{equation}
  Take $\lambda^h\in X^h$ arbitrary and define $e:= \psi - \psi^h = \eta -
  \Phi^h$, where $\Phi^h = \psi^h-\lambda^h$ and $\eta=\psi-\lambda^h$. Equation
  \eqref{eqn:ErrorEqn2} becomes
  \begin{equation}
    \begin{split}
      &a(\Phi^h,\chi^h)
        + b(\psi;\Phi^h,\chi^h) + c(\Phi^h,\chi^h) \quad = \\
      &\qquad a(\eta,\chi^h)
        + b(\psi;\eta,\chi^h) + b(\psi-\psi^H; \psi^h,\chi^h) + c(\eta,\chi^h), \quad \forall \chi^h \in X^h.
    \end{split}
    \label{eqn:ErrorSplit}
  \end{equation}
  Since \eqref{eqn:ErrorSplit} holds for any $\chi^h\in X^h$ it holds in
  particular for $\chi^h=\Phi^h\in X^h$, which implies
  \begin{equation}
    \begin{split}
      &a(\Phi^h,\Phi^h) + b(\psi;\Phi^h,\Phi^h) + c(\Phi^h,\Phi^h) = \\
        &\qquad a(\eta,\Phi^h) + b(\psi; \eta, \Phi^h)
        + b(\psi-\psi^H; \psi^h, \Phi^h) + c(\eta,\Phi^h).
    \end{split}
    \label{eqn:ErrorPhih}
  \end{equation}

  Note that $c(\psi,\chi) = -c(\chi,\psi)$, which implies $c(\Phi^h,\Phi^h) = 0$.
  Also, $b(\psi;\chi,\chi) = 0$ and so \eqref{eqn:ErrorPhih} becomes
  \begin{equation}
    \begin{split}
      a(\Phi^h,\Phi^h) &= a(\eta,\Phi^h) + b(\psi;\eta,\Phi^h)
        + b(\psi - \psi^H;\psi^h,\Phi^h) + c(\eta,\Phi^h).
    \end{split}
    \label{eqn:ErrorSimp}
  \end{equation}
  Now rewriting $b(\psi - \psi^H; \psi^h, \Phi^h)$ using \autoref{lma:trilinear}
  yields
  \begin{equation}
    \begin{split}
      a(\Phi^h,\Phi^h) &= a(\eta,\Phi^h) + b(\psi;\eta,\Phi^h) \\
       & + b_0(\psi^h; \Phi^h, \psi - \psi^H) + b_0(\Phi^h; \psi^h, \psi^H -
       \psi) + c(\eta,\Phi^h).
    \end{split}
    \label{eqn:Errorb0}
  \end{equation}
  Using the error bounds given in Lemmas \ref{lma:bounds},
  \ref{lma:Stability}, \ref{lma:Fine}, and \ref{lma:bImproved} in equation 
  \eqref{eqn:Errorb0} gives
  \begin{align*}
    Re^{-1} |\Phi^h|_2^2 &\le Re^{-1} |\eta|_2\, |\Phi^h|_2 + \Gamma_1\, |\psi|_2\, |\eta|_2\, |\Phi^h|_2 \\
      &\quad+ 2 \, C\, |\psi^h|_2\, |\Phi^h|_2\, |\psi - \psi^H|_1
      \sqrt{|\ln(h)|} + Ro^{-1}\, \Gamma_2 |\eta|_2\, |\Phi^h|_2 \\
    &= \left(Ro^{-1}\,\Gamma_2 + Re^{-1} + \Gamma_1\, |\psi|_2\right) |\eta|_2\, |\Phi^h|_2 \\
      &\qquad + 2 \, C\, |\psi^h|_2\, |\Phi^h|_2 |\psi - \psi^H|_1 \sqrt{|\ln(h)|}.
  \end{align*}
  Simplifying by $|\Phi^h|_2$ and using the stability estimates \eqref{eqn:psi}
  in \autoref{lma:Stability} and \eqref{eqn:FEStability} in
  \autoref{lma:FEStability} gives
  \begin{equation}
    \begin{split}
    |\Phi^h|_2 &\le \left(1 + Re\, Ro^{-1}\, \Gamma_2 + Re^2 Ro^{-1} \Gamma_1\,
        \|F\|_{-2}\right) |\eta|_2 \\
      &\quad + 2 Re^2 Ro^{-1} \, C\, \|F\|_{-2}\, |\psi - \psi^H|_1
        \sqrt{|\ln(h)|}.
    \end{split}
    \label{eqn:phih}
  \end{equation}
  Adding $|\eta|_2$ to both sides of \eqref{eqn:phih} and using the triangle
  inequality gives
  \begin{equation}
    \begin{split}
    |\psi - \psi^h|_2 &\le \left(2 + Re\, Ro^{-1}\, \Gamma_2 + Re^2 Ro^{-1}
        \Gamma_1\, \|F\|_{-2}\right) |\eta|_2 \\
      &\quad + 2 Re^2 Ro^{-1} \, C\, \|F\|_{-2}\, |\psi - \psi^H|_1
        \sqrt{|\ln(h)|}.
    \end{split}
    \label{eqn:psi-psih}
  \end{equation}
  Taking the infimum over $\lambda^h \in X^h$ in \eqref{eqn:psi-psih} yields the
  estimate \eqref{eqn:Error}.
\end{proof}

In what follows, we consider both $X^h$ and $X^H$ Argyris FE spaces. We
emphasize, however, that both \autoref{alg:TwoLevel} and the error estimate in
\autoref{thm:TwoLevel} remain valid for other conforming FE spaces, e.g., the Bell
element, the Hsieh-Clough-Tocher element, or the Bogner-Fox-Schmit element.

For the Argyris triangle, we have the following inequalities, which follow from
approximation theory \cite{Bernadou94} and Theorem 6.1.1 in \cite{Ciarlet}:
\begin{equation} \label{eqn:Argyris}
  \begin{split}
    |\psi - \psi^h|_j &\le Ch^{6-j}, \\
    |\psi - \psi^H|_j &\le CH^{6-j},
  \end{split}
\end{equation}
where $j=0,1,2$ and $\psi$, the solution of \eqref{eqn:WeakForm}, is assumed to
satisfy $\psi\in H^6(\Omega) \cap H^2_0(\Omega)$.
Of course, if the solution $\psi$ does not possess the assumed regularity, the convergence rates of the FE discretization can significantly degrade.
This behavior is standard for high-order numerical methods.
\begin{crlry} \label{crl:Argyris2L}
  Let $X^h,\, X^H \subset H^2_0(\Omega)$ be Argyris finite elements. Then
  $\psi^h$, the solution of the two-level algorithm (\autoref{alg:TwoLevel})
  satisfies the following error estimate:
  \begin{equation}
    |\psi - \psi^h|_2 \le C_1 h^4 + C_2 \sqrt{|\ln(h)|}\, H^5.
    \label{eqn:Argyris2L}
  \end{equation}
\end{crlry}
\begin{proof}
  This follows directly by substituting the inequalities \eqref{eqn:Argyris}
  into \eqref{eqn:Error}.
\end{proof}

  \section{Numerical Results} \label{sec:Numerical}
  The goal of this section is twofold: first, we illustrate the computational
efficiency of the two-level method, and second, we verify the theoretical rates
of convergence developed in \autoref{sec:Errors}.  To illustrate the
computational efficiency of the two-level method, we compare solution times for
the full nonlinear one-level method and for the two-level method applied to the
SQGE. We choose coarse mesh/fine mesh pairs such that the ratio is
$\nicefrac{1}{2}$. To verify the theoretical rates of convergence, we compare
the theoretical error estimates to the observed rates of convergence from our
numerical tests. 
The calculations on the coarse mesh employ our original code that was benchmarked in~\cite{Foster}.

The FE solver was written in MATLAB (the 2010b version), with the Argyris implementation written in C. 
The nonlinear systems were solved with Newton's method; at each step the resulting linear systems were solved with UMFPACK. 
The numerical tests were run on a Mac Pro with 16 gigabytes of RAM and two quad-core Intel Xeon processors. 

First, we apply the two-level method to the SQGE \eqref{eqn:Streamfunction}
with $Re=Ro=1$ and exact solution
\begin{equation}
  \psi(x,y) = \bigl(\sin \left(4\pi x \right) \sin \left( 4\pi y \right) \bigr)^2.
  \label{eqn:2LExact}
\end{equation}
The homogeneous boundary conditions are $\psi=\dfrac{\partial
\psi}{\partial \vec{n}}=0$ and the forcing function $F$ corresponds to the exact
solution \eqref{eqn:2LExact}. 

\subsection{Practical Considerations}
A key part of two-level algorithms is accessing a previous coarse mesh solution,
i.e., finding the parent element given a child element. This step can negate any
performance benefits if not implemented wisely. Indeed, let $ n $ be the number
of elements in the FE discretization. For the unit square, a na\"{i}ve search
across every element takes $ O(n/2) $ operations. This procedure may be improved
with a binary search, which is summarized in \autoref{alg:TwoLevelLookup}.

We note that every element on the fine mesh corresponds to exactly one element
on the coarse mesh. However, a coarse mesh element may correspond to multiple
elements on the fine mesh.

\begin{algorithm}[H]
  \caption{ Given an element on the fine mesh, determine the parent element on
  the coarse mesh.}
  \label{alg:TwoLevelLookup}
  Before examining the fine mesh, sort the coarse mesh elements by their
  centroid values.
  \begin{enumerate}[Step 1:]
    \item Select an element on the fine mesh and compute its centroid.
    \item Perform a binary search across the coarse mesh elements until
      the difference between the $ x $-values of the fine mesh centroid
      and coarse mesh centroids is less than $ H $, the coarse mesh step
      size. There should be many elements that fit this condition; save
      them as a list.
    \item Search through this list until we find the correct coarse mesh
      element (that is, the centroid of the fine-mesh element is an
      interior point of the correct coarse mesh element).
  \end{enumerate}
\end{algorithm}

For the considered unit square, the binary search will examine on average
$\log(n)$ elements, while the linear search component involves at most
$\sqrt{n}/2$ elements. Therefore the search requires a $O(\sqrt{n}/2)$ number of
element checks. Profiling results indicate that using
\autoref{alg:TwoLevelLookup} to identify parent elements takes much less time
than either setting up or solving the systems, so this approach is fast enough
that lookup time does not contribute significantly to overall solution time.

\subsection{Computational Efficiency}\label{sse:Efficiency}
To illustrate the computational efficiency of the two-level method, we compare
the simulation time for the standard one-level method (i.e., the full nonlinear
system, without the two-level method) with the simulation time for the two-level
method.

In \autoref{tab:Efficiency}, the $L^2$-norm of the error ($e_{L^2}$),  the
$H^1$-norm of the error ($e_{H^1}$), the
$H^2$-norm of the error ($e_{H^2}$) and the simulation times are listed for
various mesh sizes. 
For each fine mesh, we choose a coarse mesh that ensures the
same order of magnitude for the errors in the one-level and two-level methods.
For small values of the fine mesh size, $h$, the two-level method was
significantly faster than the one-level method. The errors in the $H^2$-norm
were nearly identical, while the error in the $L^2$-norm were generally of the
same order of magnitude. We also note that the tolerance in Newton's method
seems to cause a plateau in the $L^2$-norm of the error. The results in
\autoref{tab:Efficiency} are illustrated graphically in
\autoref{fig:Efficiency}. In this figure the simulation times of the one-level
method (green) and of the two-level method (blue) are displayed for all the
pairs $(h,H)$ in \autoref{tab:Efficiency}.  \autoref{fig:Efficiency} clearly
shows that as the number of degrees of freedom (DoFs) increases, the
computational efficiency of the two-level method increases as well.

\begin{table}
  \begin{center}
    {\small
    \begin{tabular}{|c|c|c|c|c|c|c|c|}
      \hline
        $H$     & $h$        & DoFs, $H$ & DoFs, $h$ & $e_{L^2}$              & $e_{H^1}$              & $e_{H^2}$             & time, $s$ \\
      \hline
      $-$       & $0.05146$  & $-$       & $4362$    & $4.286\times 10^{-8}$  & $7.767 \times 10^{-6} $ & $1.648\times 10^{-3}$ & $3.328$ \\
      $0.1083$  & $0.05146$  & $1158$    & $4362$    & $1.092\times 10^{-7}$  & $9.714 \times 10^{-6} $ & $1.709\times 10^{-3}$ & $2.372$ \\
      $-$       & $0.02561$  & $-$       & $16926$   & $5.748\times 10^{-10}$ & $2.236 \times 10^{-7} $ & $1.009\times 10^{-4}$ & $19.92$ \\
      $0.05146$ & $0.02561$  & $4362$    & $16926$   & $7.691\times 10^{-10}$ & $2.345 \times 10^{-7} $ & $1.016\times 10^{-4}$ & $11.82$ \\
      $-$       & $0.01597$  & $-$       & $43074$   & $4.751\times 10^{-11}$ & $2.688 \times 10^{8} $ & $1.793\times 10^{-5}$ & $55.69$ \\
      $0.03384$ & $0.01597$  & $10983$   & $43074$   & $5.267\times 10^{-11}$ & $2.732 \times 10^{-8} $ & $1.797\times 10^{-5}$ & $33.19$ \\
      $-$       & $0.01277$  & $-$       & $66678$   & $8.66\times 10^{-12}$  & $6.611 \times 10^{-9} $ & $6.207\times 10^{-6}$ & $102.4$ \\
      $0.02561$ & $0.01277$  & $16926$   & $66678$   & $9.686\times 10^{-12}$ & $6.676 \times 10^{-9} $ & $6.217\times 10^{-6}$ & $59.03$ \\
      $-$       & $0.009659$ & $-$       & $116614$  & $3.876\times 10^{-12}$ & $2.135 \times 10^{-9} $ & $2.382\times 10^{-6}$ & $161.7$ \\
      $0.02035$ & $0.009659$ & $29501$   & $116614$  & $6.836\times 10^{-12}$ & $2.15 \times 10^{-9}  $ & $2.385\times 10^{-6}$ & $95.93$ \\
      $-$       & $0.007959$ & $-$       & $170598$  & $4.791\times 10^{-12}$ & $7.945 \times 10^{-10}$ & $1.111\times 10^{-6}$ & $325.1$ \\
      $0.01597$ & $0.007959$ & $43074$   & $170598$  & $9.087\times 10^{-12}$ & $8.005 \times 10^{-10}$ & $1.112\times 10^{-6}$ & $172.3$ \\
      $-$       & $0.006854$ & $-$       & $230574$  & $1.79\times 10^{-11}$  & $4.109 \times 10^{-10}$ & $6.16 \times 10^{-7}$ & $401.7$ \\
      $0.01436$ & $0.006854$ & $58131$   & $230574$  & $1.3\times 10^{-11}$   & $4.138 \times 10^{-10}$ & $6.163\times 10^{-7}$ & $219.5$ \\
      $-$       & $0.006374$ & $-$       & $264678$  & $3.912\times 10^{-11}$ & $3.412 \times 10^{-10}$ & $3.846\times 10^{-7}$ & $559.7$ \\
      $0.01277$ & $0.006374$ & $66678$   & $264678$  & $2.309\times 10^{-11}$ & $2.766 \times 10^{-10}$ & $3.848\times 10^{-7}$ & $291.9$ \\
      $-$       & $0.005264$ & $-$       & $389994$  & $3.85\times 10^{-11}$  & $2.417 \times 10^{-10}$ & $2.086\times 10^{-7}$ & $753.4$ \\
      $0.01101$ & $0.005264$ & $98133$   & $389994$  & $6.495\times 10^{-11}$ & $4.156 \times 10^{-10}$ & $2.087\times 10^{-7}$ & $397.7$ \\
      \hline
    \end{tabular}}
  \end{center}
  \caption{Comparison of one-level and two-level methods: the $L^2$-norm of the
    error ($e_{L^2}$), the $ H^1 $-norm of the error ($ E_{H^1} $), and the
    $H^2$-norm of the error ($e_{H^2}$) with simulation times.}
  \label{tab:Efficiency}
\end{table}

\begin{center}
  \begin{figure}
    \begin{center}
    \includegraphics[width=4in,natwidth=610,natheight=642]{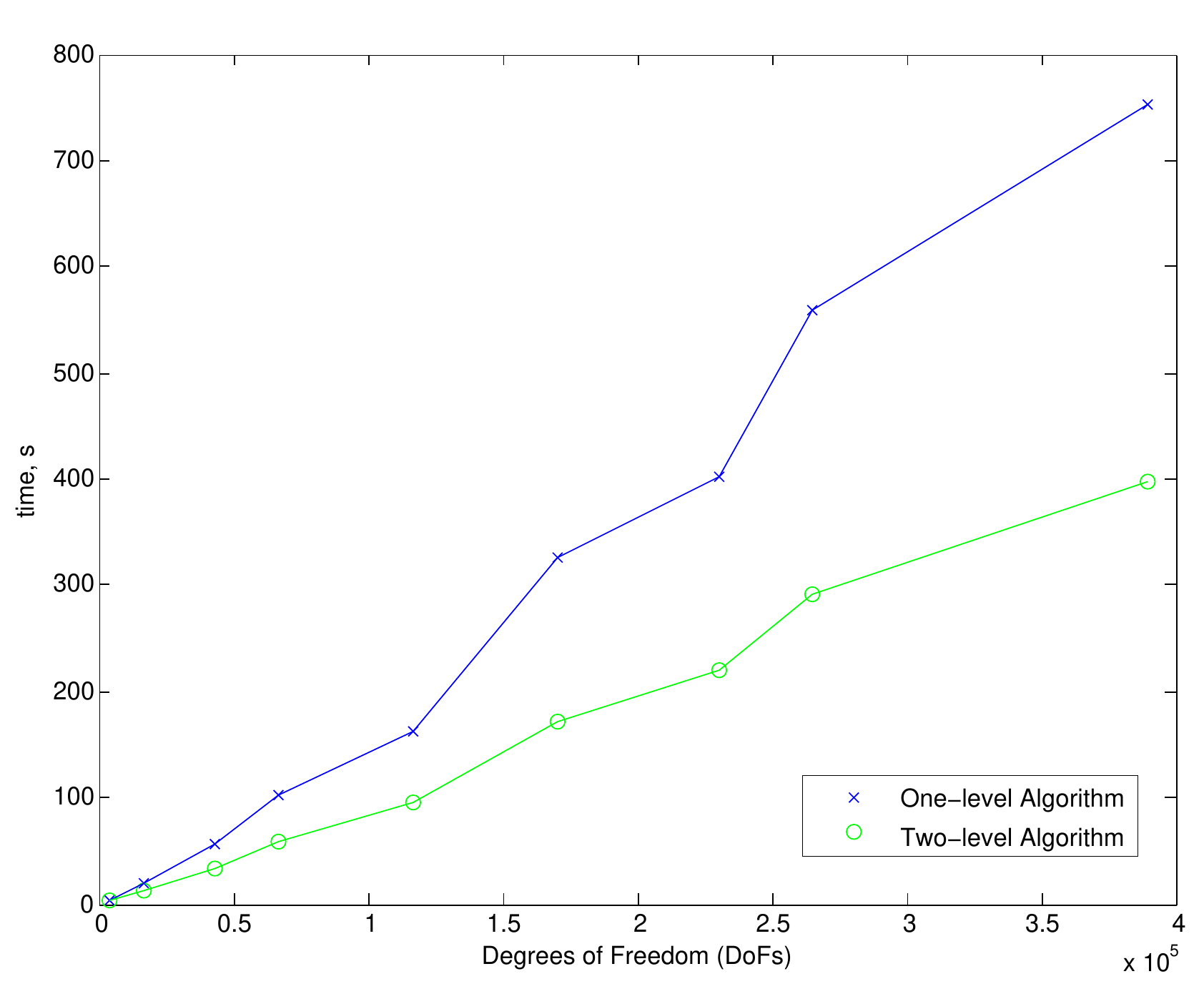}
    \includegraphics[width=4in,natwidth=610,natheight=642]{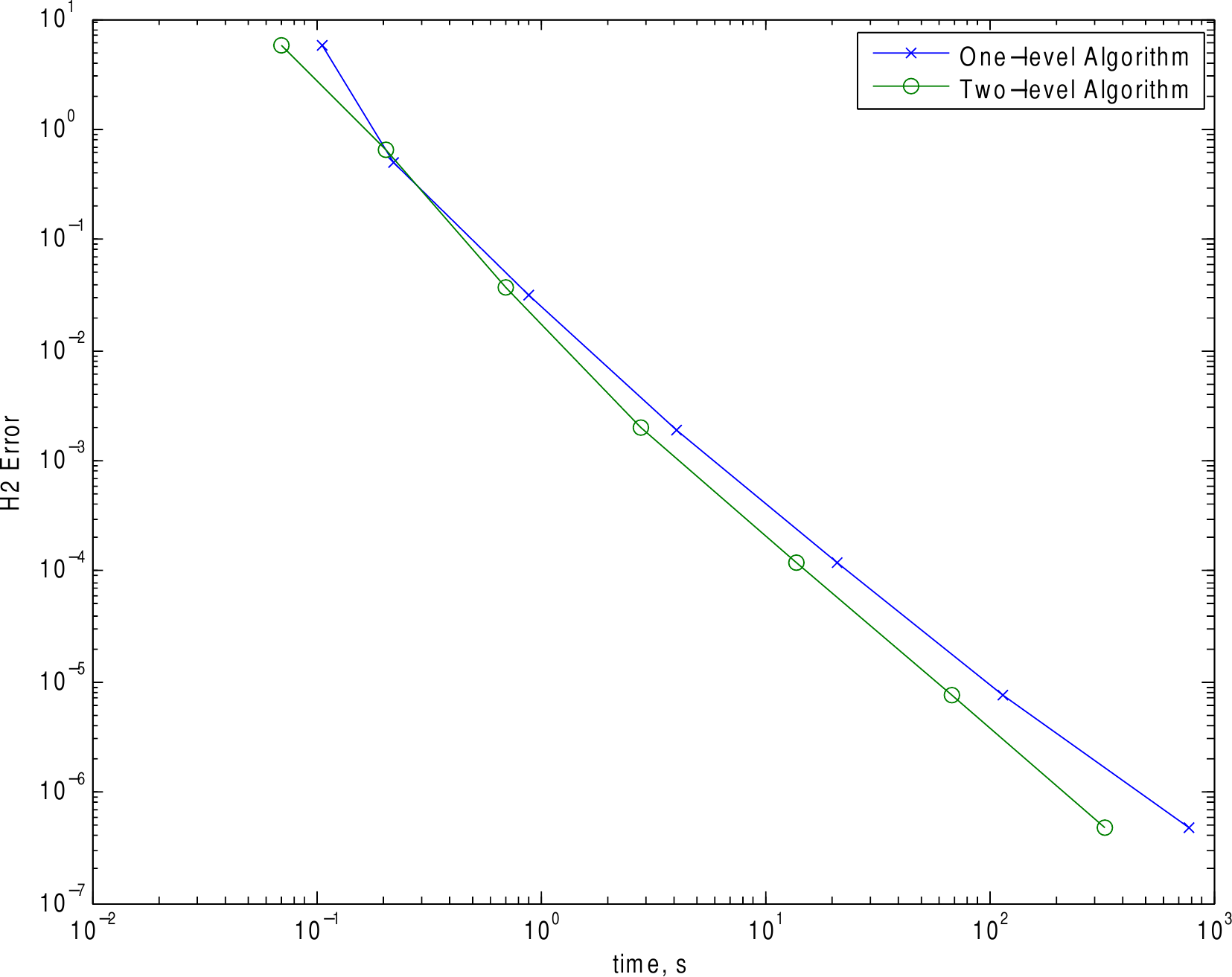}
    \caption{
    	The simulation times of the one-level method (green) and of the
      two-level method (blue) are displayed for all the pairs $(h,H)$ in
      \autoref{tab:Efficiency}.
      Top: simulation time as a function of the number of degrees of freedom.
      Bottom: log-log plot of the $H^2$-error as a function of the simulation time.
      }
    \label{fig:Efficiency}
    \end{center}
  \end{figure}
\end{center}

\subsection{Rates of Convergence}\label{sse:Rates}
The goal of this subsection is to numerically verify the theoretical rates of
convergence in estimate \eqref{eqn:Argyris2L} of Corollary \ref{crl:Argyris2L}. Unlike
the theoretical error estimates for the one-level method developed in
\cite{Foster}, for the two-level method we must verify rates of convergence
for two different mesh sizes: the fine mesh size, $h$, and the coarse mesh size, $H$.

To numerically verify the theoretical rate of convergence given in estimate \eqref{eqn:Argyris2L} with respect to $H$, we fix $h$ to a small value and we vary $H$.
Thus, the total error in estimate \eqref{eqn:Argyris2L} will be dominated by the $H$
term, i.e., the total error will be of order $O(H^5)$. In \autoref{tab:TwoLevelH},
we fix $h=0.0063$ and we vary $H$. The error in the $L^2$-norm ($e_{L^2}$), the
error in the $H^1$-norm ($e_{H^1}$), the
error in the $H^2$-norm ($e_{H^2}$), and the rate of convergence with respect to
$H$ are listed in \autoref{tab:TwoLevelH}. The rate of convergence with respect to $H$ of the error in the $H^2$-norm follows the
theoretical rate predicted in estimate \eqref{eqn:Argyris2L} (i.e., fifth-order). For the
last mesh pair, however, the rate of convergence appears to drop off. This
occurs because, for small values of $H$, the total error in
estimate \eqref{eqn:Argyris2L} is not dominated anymore by the $H$ term.
We note that in \autoref{tab:TwoLevelH} the number of $DoFs$ for the fine mesh
varies even though the mesh size $h$ is constant. The reason for this is that for the coarse
mesh a Delauney triangulation is performed around a grid spacing corresponding
to $H$ and then a red refinement is performed to obtain the fine mesh corresponding to $h$. 
Thus, differences in the Delauney triangulations for various mesh sizes $H$ yield differences in the fine meshes of mesh size $h$.

To numerically verify the theoretical rate of convergence given in estimate \eqref{eqn:Argyris2L} with respect to $h$, we must proceed with caution. The reason is
that a straightforward approach would fix $H$ and let $h$ go to zero. This
approach, however, would fail, since the $H$ term would eventually dominate the
total error. To avoid this, we consider the following scaling between the mesh
sizes:
\begin{equation}
  H = C\, h,
  \label{eqn:C-Scaling}
\end{equation}
where $C>1$. The scaling in \eqref{eqn:C-Scaling} implies that the total error
in estimate \eqref{eqn:Argyris2L} is of order $O(h^4)$. Indeed, the second term on the
right hand side of estimate \eqref{eqn:Argyris2L} now scales as follows:
\begin{equation}
  \begin{split}
    C_2 \sqrt{|\ln(h)|}\, H^5 &\approx C_2 C \sqrt{|\ln(h)|}\, h^5 
    \approx O(h^4),
  \end{split}
  \label{eqn:Balance}
\end{equation}
where in the last relation in \eqref{eqn:Balance} we used the fact that
$\sqrt{|\ln(h)|}\, h \rightarrow 0$ when $h \rightarrow 0$ (which follows from
l'Hospital's rule).

\begin{remark}
  We emphasize that the scaling in \eqref{eqn:C-Scaling} is not needed in the
  two-level algorithm. We only use it in this subsection to monitor the
  convergence rate with respect to $h$.
\end{remark}

In this subsection, we consider $C=2$ in \eqref{eqn:C-Scaling}. We note,
however, that any other constant $C>1$ could be used in \eqref{eqn:C-Scaling}.
With this choice, we are now ready to numerically verify the theoretical rate
of convergence  given in estimate \eqref{eqn:Argyris2L} with respect to $h$, which, as
shown in \eqref{eqn:Balance}, will be of order $O(h^4)$. In
\autoref{tab:TwoLevelh}, for various mesh size pairs $(H=2h, h)$, we list the
$L^2$-norm of the error ($e_{L^2}$), the $H^1$-norm of the error ($e_{H^1}$), 
the $H^2$-norm of the error ($e_{H^2}$),
and the rate of convergence. The rate of convergence with respect to $h$ of the error in the $H^2$-norm follows the theoretical
rate predicted in estimate \eqref{eqn:Argyris2L} (i.e., fourth-order).

\begin{table}
  {\tiny
  \begin{center}
    \begin{tabular}{|c|c|c|c|c|c|c|c|c|c|}
    \hline
    $H$ &   $h$  &  DoFs, $ H $ & DoFs, $ h $ & $e_{L^2}$ & $L^2$ order
      & $e_{H^1}$ & $H^1$ order & $e_{H^2}$ & $H^2$ order \\
    \hline
    $\nicefrac{1}{2}$   & $\nicefrac{1}{256}$   & $106$     & $1,183,238$
        & $1.20\times 10^{-2}$  & $-$     & $1.60\times 10^{-1}$  & $-$
        & $5.18\times 10^{0}$   & $-$ \\
    $\nicefrac{1}{4}$   & $\nicefrac{1}{256}$   & $350$     & $740,870$
        & $2.53\times 10^{-2}$  & $-1.08$ & $4.75\times 10^{-1}$  & $-1.57$
        & $11.7\times 10^{1}$   & $-1.17$ \\
    $\nicefrac{1}{8}$   & $\nicefrac{1}{256}$  & $838$     & $851,462$
        & $6.67\times 10^{-4}$  & $5.24$  & $1.36\times 10^{-2}$  & $5.12$
        & $7.84\times 10^{-1}$  & $3.89$ \\
    $\nicefrac{1}{16}$  & $\nicefrac{1}{256}$  & $3,542$   & $782,342$
        & $5.54\times 10^{-6}$  & $6.91$  & $2.32\times 10^{-4}$  & $5.87$
        & $2.79\times 10^{-2}$  & $4.81$ \\
    $\nicefrac{1}{32}$  & $\nicefrac{1}{256}$  & $12,622$  & $769,094$
        & $1.70\times 10^{-8}$  & $8.35$  & $3.38\times 10^{-6}$  & $6.10$
        & $8.41\times 10^{-4}$  & $5.05$ \\
    $\nicefrac{1}{64}$  & $\nicefrac{1}{256}$ & $48,746$  & $778,742$
        & $9.38\times 10^{-11}$ & $7.50$  & $3.82\times 10^{-8}$  & $6.47$
        & $2.31\times 10^{-5}$  & $5.18$ \\
    $\nicefrac{1}{128}$ & $\nicefrac{1}{256}$ & $195,586$ & $781,766$
        & $1.63\times 10^{-10}$ & $-0.80$ & $4.28\times 10^{-9}$  & $3.16$
        & $1.16\times 10^{-5}$  & $1.00$ \\
      \hline
    \end{tabular}
  \end{center}
  \caption{Two-level method: the $L^2$-norm of the error ($e_{L^2}$), the
  $H^2$-norm of the error ($e_{H^2}$), and the convergence rate with respect to
  $H$.}
  \label{tab:TwoLevelH}
  \begin{center}
    \begin{tabular}{|c|c|c|c|c|c|c|c|c|c|}
    \hline
    $H$ &   $h$  &  DoFs, $ H $ & DoFs, $ h $ & $e_{L^2}$ & $L^2$ order
      & $e_{H^1}$ & $H^1$ order & $e_{H^2}$ & $H^2$ order \\
    \hline
    $\nicefrac{1}{2}$   & $\nicefrac{1}{4}$   & $106$     & $350$
        & $2.58\times 10^{-2}$  & $-$     & $9.65\times 10^{-1}$  & $-$
        & $5.04\times 10^{1}$   & $-$ \\
    $\nicefrac{1}{4}$   & $\nicefrac{1}{8}$   & $350$     & $838$
        & $2.57\times 10^{-2}$  & $0.01$  & $5.63\times 10^{-1}$  & $0.78$
        & $2.54\times 10^{1}$   & $1.04$ \\
    $\nicefrac{1}{8}$   & $\nicefrac{1}{16}$  & $838$     & $3,542$
        & $6.64\times 10^{-4}$  & $5.27$  & $1.44\times 10^{-2}$  & $5.29$
        & $1.15\times 10^{0}$   & $4.42$ \\
    $\nicefrac{1}{16}$  & $\nicefrac{1}{32}$  & $3,542$   & $12,622$
        & $5.57\times 10^{-6}$  & $6.90$  & $2.89\times 10^{-4}$  & $5.64$
        & $6.04\times 10^{-2}$  & $4.25$ \\
    $\nicefrac{1}{32}$  & $\nicefrac{1}{64}$  & $12,622$  & $48,746$
        & $1.88\times 10^{-8}$  & $8.21$  & $5.89\times 10^{-6}$  & $5.62$
        & $3.36\times 10^{-3}$  & $4.17$ \\
    $\nicefrac{1}{64}$  & $\nicefrac{1}{128}$ & $48,746$  & $195,586$
        & $1.37\times 10^{-10}$ & $7.10$  & $1.36\times 10^{-7}$  & $5.44$
        & $1.83\times 10^{-4}$  & $4.20$ \\
    $\nicefrac{1}{128}$ & $\nicefrac{1}{256}$ & $195,586$ & $781,766$
        & $2.06\times 10^{-10}$ & $-0.58$ & $4.35\times 10^{-9}$  & $4.96$
        & $1.16\times 10^{-5}$  & $3.98$ \\
    \hline
    \end{tabular}
  \end{center}
  \caption{Two-level method: the $L^2$-norm of the error ($e_{L^2}$), the
  $H^2$-norm of the error ($e_{H^2}$), and the convergence rate with respect to
  $h$.}
  \label{tab:TwoLevelh}
}
\end{table}

To investigate the performance of the two-level method in a more realistic setting, we consider the same setting as that used in Test 2 in~\cite{Cascon}.
The exact solution is given by
$$ \psi(x,y) = \left(\left(1 - \dfrac{x}{3}\right) \left(1 - e^{-20 x}\right)
       \sin\left(\pi y\right)\right)^2, $$
where the computational domain is the $ [0, 3] \times [0, 1] $ rectangle.
The exact solution displays a sharp boundary layer, which is similar to the western boundary layer encountered in the large scale ocean circulation in the northern hemisphere~\cite{Cascon,Vallis06,Foster}.
Furthermore, we use the parameter values $ Re = 5 $ and $ Ro = 10^{-4}$, which are similar to the realistic values used in \cite{delSastre04,Galan-del-Sastre2004,Galan-del-Sastre2008,Foster} for the Mediterranean Sea.

In \autoref{tab:TwoLevelh_test6}, for various mesh size pairs $(H \approx 2h, h)$, we list the
$L^2$-norm of the error ($e_{L^2}$), the $H^1$-norm of the error ($e_{H^1}$), 
the $H^2$-norm of the error ($e_{H^2}$),
and the rate of convergence. The rate of convergence with respect to $h$ of the error in the $H^2$-norm follows the theoretical
rate predicted in estimate \eqref{eqn:Argyris2L} (i.e., fourth-order).

\begin{table}
{\tiny
  \begin{center}
    \begin{tabular}{|c|c|c|c|c|c|c|c|c|c|}
    \hline
    $H$                    & $h$                      & DoFs, $H$ & DoFs, $h$ & $e_{L^2}$                & $L^2$ order & $e_{H^1}$               & $H^1$ order & $e_{H^2}$               & $H^2$ order \\
    \hline
    $1.9597 \times 10^{-1}$ & $9.7985 \times 10^{-2}$ & $1001$    & $3706$    & $4.8400 \times 10^{-4}$  & $-$         & $2.4799 \times 10^{-2}$ & $-$         & $2.1768$                & $-$      \\
    $6.3424 \times 10^{-2}$ & $3.1712 \times 10^{-2}$ & $8517$    & $33210$   & $2.2665 \times 10^{-6}$  & $4.7547$    & $2.5353 \times 10^{-4}$ & $4.0626$    & $5.5286 \times 10^{-2}$ & $3.2560$ \\
    $3.8175 \times 10^{-2}$ & $1.9087 \times 10^{-2}$ & $23053$   & $90794$   & $3.3150 \times 10^{-7}$  & $3.7866$    & $4.9391 \times 10^{-5}$ & $3.2220$    & $1.3623 \times 10^{-2}$ & $2.7592$ \\
    $2.7344 \times 10^{-2}$ & $1.3672 \times 10^{-2}$ & $44573$   & $176314$  & $2.1466 \times 10^{-8}$  & $8.2030$    & $4.8908 \times 10^{-6}$ & $6.9301$    & $2.6547 \times 10^{-3}$ & $4.9011$ \\
    $2.1265 \times 10^{-2}$ & $1.0633 \times 10^{-2}$ & $73158$   & $290094$  & $3.0304 \times 10^{-9}$  & $7.7868$    & $1.1985 \times 10^{-6}$ & $5.5934$    & $9.6684 \times 10^{-4}$ & $4.0174$ \\
    $1.7385 \times 10^{-2}$ & $8.6926 \times 10^{-3}$ & $109051$  & $433106$  & $5.6666 \times 10^{-10}$ & $8.3230$    & $3.4166 \times 10^{-7}$ & $6.2297$    & $3.8517 \times 10^{-4}$ & $4.5686$ \\
    \hline
    \end{tabular}
  \end{center}
  \caption{Two-level method: the $L^2$-norm of the error ($e_{L^2}$), the
  $H^2$-norm of the error ($e_{H^2}$), and the convergence rate with respect to
  $h$.}
  \label{tab:TwoLevelh_test6}
  }
\end{table}

  \section{Conclusions} \label{sec:Conclusions}
  In this paper, we proposed a two-level FE discretization of the (nonlinear)
stationary QGE in the pure streamfunction formulation. The two-level algorithm consists of two
steps. In the first step, the nonlinear system is solved on a coarse mesh. In
the second step, the nonlinear system is linearized around the approximation
found in the first step, and the resulting linear system is solved on the fine
mesh.

Rigorous error estimates for the two-level FE discretization were derived. These
estimates are optimal in the following sense: for an appropriately chosen
scaling between the coarse mesh size, $H$, and the fine mesh size, $h$, the error in the
two-level method is of the same order as the error in the standard one-level
method (i.e., solving the nonlinear system directly on the fine mesh).

Numerical experiments for the two-level algorithm with the Argyris element
were also carried out. The numerical results verified the theoretical error
estimates, both with respect to the coarse mesh size, $H$, and the fine mesh
size, $h$. 
Furthermore, the numerical results showed that, for an appropriate scaling between the coarse and fine mesh sizes, the two-level method significantly decreases the computational time of the standard one-level method.

We plan to extend this study in several directions.
We will treat the case of multiply connected domains~\cite{Myers,Gunzburger89,gunzburger1988finite,gunzburger1988onfinite} in order to allow the numerical investigation of more realistic computational domains (e.g., islands in the Mediterranean Sea and in the North Atlantic).
We will  also consider the time-dependent QGE and the two-layer QGE (which will allow the study of stratification effects).
Finally, we plan to investigate various preconditioning techniques to improve the performance of the linear solvers used in this report (see~\cite{elman2005finite} for the NSE in the primitive variable formulation and \cite{fairag2012block} for the NSE in the streamfunction-vorticity formulation).

  \bibliographystyle{plain}
  \bibliography{QGE,comprehensive_bibliography}

\begin{thebibliography}{10}

\bibitem{Argyris}
J.~H. Argyris, I.~Fried, and D.~W. Scharpf.
\newblock The {TUBA} family of plate elements for the matrix displacement
  method.
\newblock {\em Aero. J.}, 72:701--709, 1968.

\bibitem{Auricchio2007}
F.~Auricchio, L.~Beir\~ao~de Veiga, A.~Buffa, C.~Lovadina, A.~Reali, and
  G.~Sangalli.
\newblock A fully ``locking-free'' isogeometric approach for plane linear
  elasticity problems: a stream function formulation.
\newblock {\em Comp. Meth. in Appl. Mech. and Eng.}, 197:160--172, 2007.

\bibitem{barcilon1988existence}
V.~Barcilon, P.~Constantin, and E.~S. Titi.
\newblock Existence of solutions to the {S}tommel-{C}harney model of the {G}ulf
  {S}tream.
\newblock {\em SIAM J. Math. Anal.}, 19(6):1355--1364, 1988.

\bibitem{delSastre04}
R.~Bermejo and P.~Gal\'an~del Sastre.
\newblock Long-term behavior of the wind stress circulation of a numerical
  {N}orth {A}tlantic ocean circulation model.
\newblock {\em ECCOMAS}, 2004.

\bibitem{Bernadou94}
M.~Bernadou.
\newblock Straight and curved finite elements of class {$C^1$} and some
  applications to thin shell problems.
\newblock In {\em Finite element methods ({J}yv\"askyl\"a, 1993)}, volume 164
  of {\em Lecture Notes in Pure and Appl. Math.}, pages 63--77. Dekker, New
  York, 1994.

\bibitem{Borggaard08}
J.~Borggaard, T.~Iliescu, H.~Lee, J.~P. Roop, and H.~Son.
\newblock A two-level discretization method for the {S}magorinsky model.
\newblock {\em Multiscale Modeling \& Simulation}, 7(2):599--621, 2008.

\bibitem{Borggaard12}
J.~Borggaard, T.~Iliescu, and J.~P. Roop.
\newblock Two-level discretization of the {N}avier-{S}tokes equations with
  r-{L}aplacian subgridscale viscosity.
\newblock {\em Num. Meth. P.D.E.s}, 28(3):1056--1078, 2012.

\bibitem{Braess}
D.~Braess.
\newblock {\em Finite elements: {T}heory, fast solvers, and applications in
  solid mechanics}.
\newblock Cambridge University Press, 2001.

\bibitem{Brenner}
S.~C. Brenner and L.~R. Scott.
\newblock {\em The Mathematical Theory of Finite Element Methods}.
\newblock Springer, third edition, 2008.

\bibitem{Cascon}
J.~M. Cascon, G.~C. Garcia, and R.~Rodriguez.
\newblock A priori and a posteriori error analysis for a large-scale ocean
  circulation finite element model.
\newblock {\em Comp. Meth. Appl. Mech. Eng.}, 192(51-52):5305--5327, 2003.

\bibitem{Cayco86}
M.~E. Cayco and R.~A. Nicolaides.
\newblock Finite element technique for optimal pressure recovery from stream
  function formulation of viscous flows.
\newblock {\em Math. of Comp.}, 46(174), 1986.

\bibitem{Ciarlet}
P.~G. Ciarlet.
\newblock {\em The finite element method for elliptic problems}.
\newblock North-Holland, 1978.

\bibitem{Cummins}
P.~F. Cummins.
\newblock Inertial gyres in decaying and forced geostrophic turbulence.
\newblock {\em J. Mar. Res.}, 50(4):545--566, 1992.

\bibitem{Dijkstra05}
H.~E. Dijkstra.
\newblock {\em Nonlinear physical oceanography: {A} dynamical systems approach
  to the large scale ocean circulation and el Nino}, volume~28.
\newblock Springer Verlag, 2005.

\bibitem{Dominguez08}
V.~Dominguez and F.~J. Sayas.
\newblock Algorithm 884: {A} simple {M}atlab implementation of the {A}rgyris
  element.
\newblock {\em ACM Transaction on Mathematical Software}, 35(2), 2008.

\bibitem{elman2005finite}
H.~C. Elman, D.~J. Silvester, and A.~J. Wathen.
\newblock {\em Finite elements and fast iterative solvers: with applications in
  incompressible fluid dynamics}.
\newblock Oxford University Press, 2005.

\bibitem{Fairag98}
F.~Fairag.
\newblock A two-level finite-element discretization of the stream function form
  of the {N}avier-{S}tokes equations.
\newblock {\em Comp. Math. Applic.}, 36(2):117--127, 1998.

\bibitem{Fairag03}
F.~Fairag.
\newblock Numerical computations of viscous, incompressible flow problems using
  a two-level finite element method.
\newblock {\em SIAM J. Sci. Comp.}, 24(6):1919--1929, 2003.

\bibitem{Fairag}
F.~Fairag and N.~Almulla.
\newblock Finite element technique for solving the stream function form of a
  linearized {N}avier-{S}tokes equations using {A}rgyris element.
\newblock {\em Arxiv preprint math/0406070}, 2004.

\bibitem{fairag2012block}
F.~Fairag and A.~J. Wathen.
\newblock A block preconditioning technique for the streamfunction-vorticity
  formulation of the {N}avier-{S}tokes equations.
\newblock {\em Num. Meth. P.D.E.s}, 28(3):888--898, 2012.

\bibitem{Foster}
E.~L. Foster, T.~Iliescu, and Z.~Wang.
\newblock A finite element discretization of the streamfunction formulation of
  the stationary quasi-geostrophic equations of the ocean.
\newblock {\em Comp. Meth. in Appl. Mech. and Eng.}, 2013.
\newblock Accepted.

\bibitem{Galan-del-Sastre2004}
P.~Gal\'an~del Sastre.
\newblock {\em Estudio Num\'erico del Atractor en Ecuaciones de
  {N}avier-{S}tokes Aplicadas a Modelos de Circulaci\'on del Oc\'eano}.
\newblock PhD thesis, Universidad Complutense de Madrid, Madrid, 2004.

\bibitem{Galan-del-Sastre2008}
P.~Gal\'an~del Sastre and R.~Bermejo.
\newblock Error estimates of proper orthogonal decomposition eigenvectors and
  {G}alerkin projection for a general dynamical system arising in fluid models.
\newblock {\em Numerische Mathematik}, 110(1):49--81, June 2008.

\bibitem{Girault79}
V.~Girault and P.~A. Raviart.
\newblock {\em Finite element approximation of the {N}avier-{S}tokes
  equations}.
\newblock Volume 749 of Lecture Notes in Mathematics. Springer-Verlag, 1979.

\bibitem{Girault86}
V.~Girault and P.~A. Raviart.
\newblock {\em Finite element methods for {N}avier-{S}tokes equations: theory
  and algorithms}, volume~5 of {\em Springer Series in Computational
  Mathematics}.
\newblock Springer-Verlag, 1986.

\bibitem{Gunzburger89}
M.~D. Gunzburger.
\newblock {\em Finite element methods for viscous incompressible flows}.
\newblock Computer Science and Scientific Computing. Academic Press Inc, 1989.
\newblock A Guide to Theory, Practice, and Algorithms.

\bibitem{gunzburger1988finite}
M.~D. Gunzburger and J.~S. Peterson.
\newblock Finite-element methods for the streamfunction-vorticity equations:
  Boundary-condition treatments and multiply connected domains.
\newblock {\em SIAM J. Sci. Stat. Comput.}, 9:650--668, 1988.

\bibitem{gunzburger1988onfinite}
M.~D. Gunzburger and J.~S. Peterson.
\newblock On finite element approximations of the streamfunction-vorticity and
  velocity-vorticity equations.
\newblock {\em Int. J. Numer. Methods Fluids}, 8(10):1229--1240, 1988.

\bibitem{Johnson}
C.~Johnson.
\newblock {\em Numerical solution of partial differential equations by the
  finite element method}, volume~32.
\newblock Cambridge University Press, New York, 1987.

\bibitem{Layton93}
W.~Layton.
\newblock A two-level discretization method for the {N}avier-{S}tokes
  equations.
\newblock {\em Comp. Math. Applic.}, 26(2):33--38, 1993.

\bibitem{Ye98}
W.~Layton and X.~Ye.
\newblock Nonconforming two-level discretization of stream function form of the
  {N}avier-{S}tokes equations.
\newblock {\em Appl. Math. \& Comp.}, 89:173--183, 1998.

\bibitem{Ye99NFAO}
W.~Layton and X.~Ye.
\newblock Two level discretization of the stream functions form of the
  {N}avier-{S}tokes equations.
\newblock {\em Num. Func. Anal. \& Opt.}, 20:909--916, 1999.

\bibitem{layton2008introduction}
W.~J. Layton.
\newblock {\em Introduction to the numerical analysis of incompressible viscous
  flows}, volume~6.
\newblock Society for Industrial and Applied Mathematics, 2008.

\bibitem{Lenferink94}
H.~W.~J. Lenferink.
\newblock An accurate solution procedure for fluid flow with natural
  convection.
\newblock {\em Numer. Funct. Anal. Optim.}, 15:661--687, 1994.

\bibitem{Majda}
A.~J. Majda and X.~Wang.
\newblock {\em Non-linear dynamics and statistical theories for basic
  geophysical flows}.
\newblock Cambridge University Press, 2006.

\bibitem{Myers}
P.~G. Myers and A.~J. Weaver.
\newblock A diagnostic barotropic finite-element ocean circulation model.
\newblock {\em J. Atmos. Oceanic Technol.}, 12:511, 1995.

\bibitem{San12}
O.~San, A.~E. Staples, and T.~Iliescu.
\newblock Approximate deconvolution large eddy simulation of a stratified
  two-layer quasigeostrophic ocean model.
\newblock {\em Ocean Modelling}, 63:1--20, 2013.

\bibitem{San11}
O.~San, A.~E. Staples, T.~Iliescu, and Z.~Wang.
\newblock Approximate deconvolution large eddy simulation of a barotropic ocean
  circulation model.
\newblock {\em Ocean Modelling}, 40(2):120--132, 2011.

\bibitem{Shao11}
X.~Shao and D.~Han.
\newblock A two-grid algorithm based on {N}ewton iteration for the stream
  function form of the {N}avier-{S}tokes equations.
\newblock {\em Appl. Math. J. Chinese Univ.}, 26(3):368--378, 2011.

\bibitem{Temam}
R.~Temam.
\newblock {\em {N}avier-{S}tokes Equations: Theory and Numerical Analysis}.
\newblock North-Holland, 1984.

\bibitem{thomee2006galerkin}
V.~Thom{\'e}e.
\newblock {\em {Galerkin finite element methods for parabolic problems}}.
\newblock Springer Verlag, 2006.

\bibitem{Vallis06}
G.~K. Vallis.
\newblock {\em Atmosphere and ocean fluid dynamics: {F}undamentals and
  large-scale circulation}.
\newblock Cambridge University Press, 2006.

\bibitem{wolansky1988existence}
G.~Wolansky.
\newblock Existence, uniqueness, and stability of stationary barotropic flow
  with forcing and dissipation.
\newblock {\em Comm. Pure Appl. Math.}, 41(1):19--46, 1988.

\bibitem{Xu94}
J.~Xu.
\newblock A novel two-grid method for semilinear elliptic equations.
\newblock {\em SIAM J. on Sci. Comp.}, 15(1):231--237, 1994.

\bibitem{Ye99AMC}
X.~Ye.
\newblock Two grid discretization with backtracking of the stream function form
  of the {N}avier-{S}tokes equations.
\newblock {\em Appl. Math. \& Comp.}, 100:131--138, 1999.

\end{thebibliography}
\end{document}